\documentclass[12pt,reqno,a4paper]{amsart}

\usepackage{amsmath,amsfonts,amsthm,mathrsfs,amssymb,cite}
\usepackage[usenames]{color}
\usepackage{graphicx}
\usepackage{latexsym} 
\usepackage{enumerate}

\theoremstyle{plain}
\newtheorem{theorem}{Theorem}[section]
\newtheorem*{theorem*}{Theorem}

\newtheorem{proposition}[theorem]{Proposition}
\newtheorem*{proposition*}{Proposition}

\newtheorem*{conjecture*}{Conjecture}

\theoremstyle{definition}
\newtheorem{definition}[theorem]{Definition}

\theoremstyle{remark}
\newtheorem{remark}[theorem]{Remark}

\DeclareMathOperator{\supp}{supp}

\newcommand{\abs}[1]{\left\lvert #1 \right\rvert}
\newcommand{\norm}[1]{\left\lVert #1 \right\rVert}

\newcommand{\R}{\mathbb{R}}
\newcommand{\C}{\mathbb{C}}
\newcommand{\Z}{\mathbb{Z}}
\newcommand{\N}{\mathbb{N}}

\renewcommand{\eqref}[1]{\textnormal{(\ref{#1})}}

\numberwithin{equation}{section}

\newcommand{\polygon}{\Omega}
\newcommand{\cone}{{\mathfrak C}}

\title[Transmission eigenfunctions vanishing around corners]{On vanishing
  near corners of transmission eigenfunctions}

\author{Eemeli Bl{\aa}sten}

\address{Jockey Club Institute for Advanced Study, Hong Kong
  University of Science and Technology, Hong Kong SAR.}

\email{iaseemeli@ust.hk}

\author{Hongyu Liu}
\address{Department of Mathematics, Hong Kong Baptist University, Kowloon Tong, Hong Kong SAR.\vspace*{-4mm}}
\address{\vspace*{-4mm}and}
\address{HKBU Institute of Research and Continuing Education, Virtual University Park, Shenzhen, P. R. China.}
\email{hongyu.liuip@gmail.com}

\begin{document}

\begin{abstract}
Let $\Omega$ be a bounded domain in $\mathbb{R}^n$, $n\geq 2$, and $V\in L^\infty(\Omega)$ be a potential function. Consider the following transmission eigenvalue problem for nontrivial $v, w\in L^2(\Omega)$ and $k\in\mathbb{R}_+$,
\[
\begin{cases}
  (\Delta+k^2)v= 0 & \hspace*{-1.25cm}\mbox{in}\ \ \polygon, \\
  (\Delta+k^2(1+V))w= 0 & \hspace*{-1.25cm} \mbox{in}\ \ \polygon, \\
  w-v \in H^2_0(\polygon), \quad \norm{v}_{L^2(\polygon)}=1. &
\end{cases}
\]
We show that the transmission eigenfunctions $v$ and $w$ carry the geometric information of $\mathrm{supp}(V)$. Indeed, it is proved that $v$ and $w$ vanish near a corner point on $\partial \Omega$ in a generic situation where the corner possesses an interior angle less than $\pi$ and the potential function $V$ does not vanish at the corner point. This is the first quantitative result concerning the intrinsic property of transmission eigenfunctions and enriches the classical spectral theory for Dirichlet/Neumann Laplacian. We also discuss its implications to inverse scattering theory and invisibility. 


  
\end{abstract}

\maketitle

\section{Introduction}\label{sec:Intro}

Let $\Omega$ be a bounded domain in $\mathbb{R}^n$, $n\geq 2$, and $V\in L^\infty(\Omega)$ be a potential function. Consider the following (interior) transmission eigenvalue problem for $v, w\in L^2(\Omega)$,
\begin{equation}\label{eq:ite1}
\begin{cases}
  (\Delta+k^2)v= 0 & \hspace*{-1.25cm}\mbox{in}\ \ \polygon, \\
  (\Delta+k^2(1+V))w= 0 & \hspace*{-1.25cm} \mbox{in}\ \ \polygon, \\
  w-v \in H^2_0(\polygon), \quad \norm{v}_{L^2(\polygon)}=1. &
\end{cases}
\end{equation}
If the system \eqref{eq:ite1} admits a pair of nontrivial solutions
$(v, w)$, then $k$ is referred to as an \emph{(interior) transmission
  eigenvalue} and $(v,w)$ is the corresponding pair of
\emph{(interior) transmission eigenfunctions}. Note in particular that
nothing is imposed a-priori on the boundary values of $v$ or $w$
individually. In this paper, we are mainly interested in the real
eigenvalues, $k\in\mathbb{R}_+$, which are physically relevant.  The
study of the transmission eigenvalue problem has a long history and is
of significant importance in scattering theory.  The transmission
eigenvalue problem is a type of non elliptic and non self-adjoint
problem, so its study is mathematically interesting and
challenging. In the literature, the existing results are mainly
concerned with the spectral properties of the transmission
eigenvalues, including the existence, discreteness and infiniteness,
and Weyl laws; see for example \cite{CKP, RS, PS, CGH, Robbiano, LV,
  BP} and the recent survey \cite{CH2013inBook}. There are few results
concerning the intrinsic properties of the transmission
eigenfunctions. Here we are aware that the completeness of the set of
\emph{generalized transmission eigenfunctions} in $L^2$ is proven in
\cite{BP, Robbiano}.

In this paper, we are concerned with the vanishing properties of
\emph{interior transmission eigenfunctions}. It is shown that in
admissible geometric situations, transmission eigenfunctions which can
be approximated suitably by Herglotz waves will vanish at corners of
the support of the potential $V$. To our best knowledge, this is the
first quantitative result on intrinsic properties of transmission
eigenfunctions. As expected, these carry geometric information of the
support of the underlying potential $V$ as well as other interesting
consequences and implications in scattering theory, which we shall
discuss in more details in Section~\ref{sect:discussion}.

The location of vanishing of eigenfunctions is an important area of
study in the classical spectral theory for the Dirichlet/Neumann
Laplacian. Two important topics are the \emph{nodal sets} and
\emph{eigenfunction localization}. The former is the set of points in
the domain where the eigenfunction vanishes. For the latter, an
eigenfunction is said to be \emph{localized} if most of its
$L^2$-energy is contained in a subdomain which is a fraction of the
total domain. Considerable effort has been spent on the nodal sets and
localization in the classical spectral theory.  We refer to the recent
survey \cite{Grebenkov-Nguyen-2013}. For the curious, we mention
briefly basic facts about them, all of which are completely open for
transmission eigenfunctions. Nodal sets are $C^\infty$-curves whose
intersections form equal angles. By the celebrated Courant's nodal
line theorem, the nodal set of the $m$-th eigenfunction divides the
domain into at most $m$ nodal domains. Localization seems to be a more
recent topic even though some examples have been known for a long
time. A such example is the \emph{whispering gallery modes} that comes
from Lord Rayleigh's study of whispering waves in the Saint Paul
Cathedral in London during the late 19th century. These eigenfunctions
concentrate their energy near the boundary of a spherical or
elliptical domain. Other well known localized modes are called
\emph{bouncing ball modes} and \emph{focusing modes}
\cite{Keller-Rubinov, Chen-Morris-Zhou}. It is worth noting that the
Laplacian does not possess localized eigenfunctions on rectangular or
equilateral triangular domains \cite{Nguyen}. However, localization
does appear for the classical eigenvalue problem in a certain sense
when the angle is reflex \cite{Nazarov}. We also refer to
\cite{Heilman--Strichartz} for more relevant examples.

In our case of the transmission eigenvalue problem, peculiar and
intriguing phenomena are observed in that both vanishing and
localization of transmission eigenfunctions may occur near corners of
the support of the potential. Indeed, in an upcoming numerical paper
\cite{BLLW}, we show that if the interior angle of a corner is less
than $\pi$, then the transmission eigenfunctions vanish near the
corner, whereas if the interior angle is bigger than $\pi$, then the
transmission eigenfunctions localize near the corner. In this paper,
we shall rigorously justify the vanishing property of the transmission
eigenfunction in a certain generic situation. It turns out to be a
highly technical matter. In fact, even in the classical spectral
theory, the intrinsic properties of the eigenfunctions are much more
difficult to study than those of the eigenvalues, and they remain a
fascinating topic for a lot of ongoing research. Nevertheless, we
would also like to mention that with the help of highly accurate
computational methods, we can present a more detailed numerical
investigation in \cite{BLLW} including the vanishing/localizing order
as well as its relationship to the angle of the corner.

We believe that the vanishing and localizing properties of
transmission eigenfunctions are closely related to the analytic
continuation of the eigenfunctions. Indeed in the recent papers
\cite{BPS, PSV, HSV, EH1, EH2}, it is shown that transmission
eigenfunctions cannot be extended analytically to a neighbourhood of a
corner. The failure of the analytic continuation of transmission
eigenfunctions can be used via an absurdity argument in \cite{HSV} to
show the uniqueness in determining the polyhedral support of an
inhomogeneous medium by a single far-field pattern in the inverse
scattering theory. By further quantifying the aforementioned analytic
continuation property of transmission eigenfunctions, sharp stability
estimates were established in \cite{BL2016} in determining the
polyhedral support of an inhomogeneous medium by a single far-field
pattern. Those uniqueness and stability results already indicate that
the intrinsic properties of transmission eigenfunctions carry
geometric information of the underlying potential function
$V$. Furthermore in \cite{BL2016}, as an interesting consequence of
the quantitative estimates involved, a sharp lower bound can be
derived for the far-field patterns of the waves scattered from
polyhedral potentials associated with incident plane waves. In this
paper, we can significantly extend this result by establishing a
similar quantitative lower bound associated with incident Herglotz
waves. On the other hand, it is known \cite{Ca} that the scattered
waves created by incident waves that are Herglotz approximations to
transmission eigenfunctions will have an arbitrarily small far-field
energy. This critical observation apparently indicates that the
transmission eigenfunctions must vanish near the corner point. We
shall give more relevant discussion of our results in
Section~\ref{sect:discussion}, connecting our study to inverse
scattering problems and invisibility cloaking.

The rest of the paper is organized as follows. We will recall
scattering theory and define notation in Section
\ref{sect:notation}. All of the background and admissibility
assumptions are contained therein. We state our main results
mathematically in Section \ref{sect:results}, and then proceed to
prove them in Section \ref{sect:ffbound} and Section
\ref{sect:vanishing} using results from Section
\ref{sect:auxResults}.

\section{Preliminaries} \label{sect:notation}

In this section we recall background theory, lay some definitions and
fix notation. We will start by describing acoustic scattering theory
for penetrable scatterers. This will be referred to as ``background
assumptions'' in theorems. After that we recall what is the interior
transmission problem and some of its known facts. Finally we define
which potentials are admissible for our theorems.

\subsection{Background assumptions} \label{subsect:background}

Whenever we say that ``let the background assumptions hold'' we mean
that everything in this section should hold, unless stated
otherwise. We will recall the fundamentals of acoustic scattering
theory. For more details in the three dimensional case we refer the
readers to \cite{CK}.

We will consider only scatterers of finite diameter that are contained
in a large origin-centered ball, the \emph{domain of interest},
\[
B_R = B(\bar0,R) = \{ x\in\R^n \mid \abs{x}<R \}
\]
where $R>1$ is fixed. Let $V \in L^\infty(B_R)$ be a bounded potential
function representing the medium parameter of the scatterer. We shall
consider scattering of a fixed frequency by fixing the wavenumber
$k\in\mathbb{R}_+$.

The scatterer $V$ is illuminated by an incident wave, which in this
paper is chosen to be any Herglotz wave. These are superpositions of
plane-waves that can be written as
\begin{equation}\label{eq:herg1}
u^i(x) = \int_{\mathbb{S}^{n-1}} e^{ik\theta\cdot x} g(\theta)
d\sigma(\theta)
\end{equation}
where the kernel $g\in L^2(\mathbb{S}^{n-1})$. We say that $u^i$ is
\emph{normalized} if $\norm{g}_{L^2(\mathbb{S}^{n-1})}=1$. The field
$u^i$ is called \emph{incident} because it satisfies the equation
\[
(\Delta+k^2)u^i=0
\]
which corresponds to a background unperturbed by the presence of $V$.

Unless $V$ is transparent to $u^i$, the illumination of $V$ by $u^i$
creates a unique scattered wave $u^s\in H^2_\text{loc}(\R^n)$ such that
\begin{equation}\label{eq:scattering2}
\begin{split}
  (\Delta + k^2(1+V))u = 0& \ \ \mbox{in}\ \ \R^n,\\ u = u^i +
  u^s,&\\ \lim_{r\to\infty} r^\frac{n-1}{2} (\partial_r u^s - i k u^s)
  = 0.
\end{split}
\end{equation}
Here $u$ is the \emph{total field} which, as a superposition of the
incident field and scattered field, represents the physical observable
field. The third condition, where $r=\abs{x}$, says that $u^s$
satisfies the \emph{Sommerfeld radiation condition}, which can be
interpreted as having $u^s$ propagating from $V$ to infinity instead
of the other way around.

A property of the scattered field is that as one zooms out, the
potential $V$ starts to look more and more like a point-source in a
sense. This means that far away, $u^s$ looks like the Green's function
to $\Delta+k^2$ but modulated by a \emph{far-field pattern}
$u^s_\infty$. More precisely, as $\abs{x}\to\infty$, $u$ has the
expansion
\[
u(x) = u^i(x) + \frac{e^{ik\abs{x}}}{\abs{x}^{(n-1)/2}}
u^s_\infty\left(\frac{x}{\abs{x}}; u^i\right) + \mathcal{O}\left(
\frac{1}{\abs{x}^{n/2}} \right)
\]
where for a fixed $u^i$ the far-field pattern is a real-analytic map
$u^s_\infty:\mathbb{S}^{n-1} \to \C$ (it is also called the
\emph{scattering amplitude}).

\subsection{The interior transmission problem} \label{subsect:ITP}
Direct scattering theory is all about the study of the map $(u^i, V)
\mapsto u^s_\infty$. Given a potential $V$ the \emph{far-field
  operator}\footnote{Also called the \emph{relative scattering
    operator}. The unitary \emph{scattering operator} is the identity
  plus the former.} maps the Herglotz kernel $g$ of $u^i$ to the
far-field pattern $u^s_\infty$. In inverse scattering one is
interested in recovering meaningful information about the scatterer
$V$ from full or partial information of the far-field operator.

A number of algorithms in inverse scattering, such as linear sampling
\cite{Colton-Kirsch} and factorization methods \cite{Kirsch-Grinberg}
fail at wavenumbers where the far-field operator has non-trivial
kernel. In such a case there is an incident wave $u^i$ for which $V$
does not cause a detectable change in the far-field, and thus by
Rellich's lemma and unique continuation $u^i$ does not scatter at all:
$\supp u^s \subset \overline{\Omega}$. If this happens we call $k$ a
\emph{non-scattering energy} (or wavenumber) and say that $V$ is
transparent to $u^i$, or that $u^i$ is non-scattering. It is known
that there are radially symmetric potentials which are transparent to
certain incident waves \cite{GH}.

If $u^i$ is non-scattering and we restrict it to the supporting set
$\Omega$, then the following \emph{interior transmission problem} has
a non-trivial solution $(v,w) \in L^2(\Omega) \times L^2(\Omega)$
\begin{align}
  (\Delta+k^2)v &= 0 \quad \mbox{in}\ \ \Omega, \label{ITP1}\\
  (\Delta+k^2(1+V))w &= 0 \quad \mbox{in}\ \ \Omega, \label{ITP2}\\
  w-v &\in H^2_0(\Omega), \label{ITP3}
\end{align}
namely $v = u^i_{|\Omega}$ and $w = u_{|\Omega}$. When this
non-elliptic, non self-adjoint eigenvalue problem has a solution, we
call $k$ a transmission eigenvalue. The functions
$v$ and $w$ are referred to as the transmission eigenfunctions.

If $V$ is radially symmetric, then $v$ in \eqref{ITP1} extends to the
whole $\R^n$ as a Herglotz function, and hence in this case
transmission eigenvalues and non-scattering energies coincide
\cite{CM}. This observation hinted for a long time that these sets of
numbers coincide in general. However it was a red herring: a series of
papers on corner scattering \cite{BPS, PSV, EH1, EH2} showed that in
the presence of a certain type of corner or edge singularity in the
potential $V$ there are no non-scattering energies despite the
well-known fact that such a scatterer always has an infinite discrete
set of transmission eigenvalues.

We remark that the problem \eqref{ITP1}--\eqref{ITP3} has been studied
heavily \cite{CH2013inBook}. Many properties of
the transmission eigenvalues are known. Despite this almost nothing is
known about the eigenfunctions themselves before this paper.

\subsection{Herglotz approximation}

We introduced the Herglotz wave function in \eqref{eq:herg1}, which shall be used to approximate the 
transmission eigenfunction $v$ satisfying \eqref{ITP1}. We briefly recall the following result concerning the Herglotz approximation for the subsequent use.

\begin{theorem}[Theorem 2 in \cite{Weck}]\label{thm:herg1}
Let $\mathbf{W}_k$ denote the space of all Herglotz wave functions of
the form \eqref{eq:herg1}. For $\Omega\subset\R^n$ a $C^0$-domain,
define
\[
\mathbf{U}_k(\Omega):=\{u\in C^\infty(\Omega); (\Delta+k^2) u=0\},
\]
and
\[
\mathbf{W}_k(\Omega):=\{u_{|\Omega}; u\in \mathbf{W}_k\}.
\]
Then $\mathbf{W}_k(\Omega)$ is dense in $\mathbf{U}_k(\Omega)\cap
L^2(\Omega)$ with respect to the topology induced by the $L^2$-norm.
\end{theorem}

\subsection{Admissible potentials} \label{subsect:admissibility}
As part of our proof of the vanishing of transmission eigenfunctions
at corners we will show lower bounds for the far-field pattern
$u^s$. That is, we shall consider the scattering from a corner and
make use of the corner singularity in the potential. To save
notational burden we collect these a-priori assumptions in this
section.

We shall only consider polygonal or hypercuboidal scatterers $V$ for
simplicity. In essence $V$ will be defined as a H\"older-continuous
function $\varphi$ restricted to a polygonal domain $\polygon$; see
below. As the arguments are local, the results will hold qualitatively
for any potential $V$ for which $V_{|\mathcal U} = \chi_\polygon
\varphi_{|\mathcal U}$ for some open set $\mathcal U$ and such that
there is a reasonable path from $\mathcal U$ to infinity.

\begin{definition}
  Recalling the notation $B_R$ from Section \ref{subsect:background},
  we say that the potential $V$ is (qualitatively) admissible if
  \begin{enumerate}
    \item $V = \chi_\polygon \varphi$, where $\chi_\polygon(x)=1$ if
      $x\in\polygon$ and $\chi_\polygon(x)=0$ otherwise;
    \item $\polygon \subset B_R$ is an open convex polygon in 2D or a
      cuboid in higher dimensions;
    \item $\varphi \in C^\alpha(\R^n)$ for some $\alpha > 0$ in 2D and
      $\alpha > 1/4$ in higher dimensions;
    \item $\varphi \neq 0$ at some vertex of $\polygon$.
  \end{enumerate}
\end{definition}

\subsection{Function order} \label{subsect:order}
An important concept in corner scattering is the so-called
\emph{function order}. This determines how flat the function is at a
certain point, or in other words what is the order of the first
non-trivial homogeneous polynomial in its Taylor expansion at that
point.

\begin{definition}
  Let $f$ be a complex-valued function defined in an open
  neighbourhood of $x_c \in \R^n$. We say that $f$
  \emph{has order $N$ at $x_c$} if
  \[
  N = \max\{ M\in\Z \mid \exists C<\infty: \abs{f(x)} \leq
  C\abs{x-x_c}^M \text{ near } x_c\}.
  \]
  If the set is unbounded from above we say that $f$ has order
  $\infty$. If the set is empty $f$ has order $-\infty$.
\end{definition}

\begin{remark}
  If $f$ is smooth then it has order $N<\infty$ at $x_c$ if and only
  if $\partial^\alpha f(x_c) = 0$ for $\alpha\in\N^n$,
  $\abs{\alpha}<N$ and $\partial^\beta f(x_c) \neq 0$ for some
  $\beta\in\N^n$, $\abs{\beta}=N$. When $N=\infty$ the second
  condition is ignored: there are smooth functions vanishing to
  infinite order e.g. $\exp(-1/\abs{x}^2)$. Smooth functions always
  have non-negative order.
\end{remark}

\section{Statement of the main results} \label{sect:results}

\begin{theorem} \label{lowerBoundThm}
  Let $n \in \{2,3\}$ and let the background assumptions hold. If $V$
  is qualitatively admissible with $\varphi(x_c)\neq0$ at a vertex
  $x_c$ of $\polygon$, and $\mathcal{N}\in\N$, then there is
  $c,\ell<\infty$ depending on $V,n,k,\mathcal{N}$ and
  $\mathcal{S}=\mathcal{S}(V,k) \geq 1$ such that
  \begin{equation} \label{FFlowerBound}
    \norm{u^s_{\infty}}_{L^2(\mathbb S^{n-1})} \geq
    \frac{\mathcal{S}}{\exp\exp \big(
      c\min(1,\norm{P_N})^{-\ell}\big)}
  \end{equation}
  for any normalized incident Herglotz wave $u^i$ which is of order $N
  \leq \mathcal{N}$ at $x_c$ and whose Taylor expansion there begins
  with $P_N$. Here $\norm{P_N} = \int_{\mathbb{S}^{n-1}}
  \abs{P_N(\theta)} d\sigma(\theta)$.
\end{theorem}

\begin{theorem} \label{vanishingThm}
  Let $n\in\{2,3\}$ and $V$ be a qualitatively admissible
  potential. Assume that $k>0$ is a transmission eigenvalue: there
  exists $v,w \in L^2(\polygon)$ such that
  \begin{align*}
    (\Delta+k^2)v &= 0 \quad \mbox{in}\ \ \polygon \\
    (\Delta+k^2(1+V))w &= 0 \quad \mbox{in}\ \ \polygon \\
    w-v &\in H^2_0(\polygon), \quad \norm{v}_{L^2(\polygon)} = 1.
  \end{align*}
  If $v$ can be approximated in the $L^2(\polygon)$-norm by a sequence
  of Herglotz waves with uniformly $L^2(\mathbb S^{n-1})$-bounded
  kernels, then
  \[
  \lim_{r\to0} \frac{1}{m(B(x_c,r))} \int_{B(x_c,r)} \abs{v(x)} dx = 0
  \]
  where $x_c$ is any vertex of $\polygon$ such that $\varphi(x_c) \neq
  0$.
\end{theorem}

\begin{remark}
  A sequence of Herglotz waves $v_j$ with univormly bounded kernels
  has uniformly bounded $L^2$-norms in any fixed bounded set. However
  the converse is not true by inspecting a sequence of spherical
  harmonics $g_j = Y_j^0$. In other words the condition we have here
  is rather technical. See Section \ref{sect:discussion} for more relevant discussion.
\end{remark}

\section{Auxiliary results} \label{sect:auxResults}

In this section, we collect three auxiliary propositions that follow
without too much effort from our previous results in \cite{BL2016}
concerning the corner scattering. We add a proposition showing that in
the presence of transmission eigenfunctions incident waves creating
arbitrary small far-field patterns can be generated. Finally, another
proposition gives a lower bound for the Laplace transform of a
harmonic polynomial. The latter is necessary for quantitative
estimates involving incident Herglotz waves in corner scattering. In
comparison, we note that the paper \cite{BL2016} is mainly concerned
with corner scattering associated with incident plane waves.

\begin{proposition}
  Let the background assumptions hold with $n\in\{2,3\}$, $V$
  qualitatively admissible, $u^i$ a normalized Herglotz wave and let
  $\mathcal S \geq 1$. Then there is $\varepsilon_m(\mathcal S, k, R)
  > 0$ such that if $\norm{u^s}_{H^2(B_{2R})} \leq \mathcal S$ and
  $\norm{u^s_\infty}_{L^2(\mathbb{S}^{n-1})} \leq \varepsilon_m$ then
  \begin{equation} \label{boundaryBound}
  \sup_{x\in \partial\polygon} \abs{u^s(x)} + \abs{\nabla u^s(x)} \leq
  \frac{c}{ \sqrt{\ln \ln \frac{\mathcal S}
      {\norm{u^s_\infty}_{L^2(\mathbb{S}^{n-1})}}} }
  \end{equation}
  for some $c = c(V, \mathcal S, k, R)<\infty$.
\end{proposition}

This is a less general version of Proposition 5.10 in our previous
paper. We will also need a ``converse'' result estimating the
far-field pattern by the near-field. In more detail, we will build
incident waves with arbitrarily small far-field patterns in the
presence of a transmission eigenfunction (cf. \cite{Ca}).

\begin{proposition} \label{prop:smallFF}
  Let the background assumptions hold with $V$ supported in
  $\overline{\Omega}$, and assume that $(v,w) \in L^2(\Omega)\times
  L^2(\Omega)$ are a pair of transmission eigenfunctions on a bounded
  domain $\Omega$. There is $C = C(V,k) < \infty$ such that if $v_j
  \in L^2_{loc}$ is an incident wave such that
  $\norm{v-v_j}_{L^2(\Omega)} < \varepsilon$ then the produced
  far-field pattern has $\norm{v^s_{j\infty}}_{L^2(\mathbb S^{n-1})} <
  C \varepsilon$.
\end{proposition}
\begin{proof}
  Let $v_0^i$ be the zero-extension of $v$ to the whole $\R^n$, and
  let $v_0^s$ be the radiating solution to $(\Delta+k^2(1+V))v_0^s =
  -k^2 V v_0^i$. Also let $\nu^s_0$ be the zero-extension of $w-v\in
  H^2_0(\Omega)$ to $\R^n$.  By standard scattering theory
  (e.g. Chapter 8 in \cite{CK}) we see that $v_0^s = \nu_0^s$ since
  \[
  (\Delta+k^2(1+V)) v_0^s = -k^2 V v_0^i = - k^2 V v =
  (\Delta+k^2(1+V)) \nu_0^s
  \]
  in $\R^n$ and both satisfy the Sommerfeld radiation condition
  trivially. Hence the far-field pattern of $v_0^s$ is zero.

  Since $v_j$ approximates $v$ in $L^2(\Omega)$, and $V$ is supported
  on $\overline{\Omega}$, we have $-k^2 V v_j$ approximating $-k^2 V
  v_0^i$ in $\R^n$. Let $v_j^s$ be the scattered wave arising from the
  incident wave $v_j$ and potential $V$. Then, again from standard
  scattering theory, its far-field pattern approximates the far-field
  pattern of $v_0^s$, i.e. zero. The operators involved are all
  bounded, so
  \begin{equation}
    \norm{ v^s_{j\infty} }_{L^2(\mathbb S^{n-1})} < C_{V,k}
    \varepsilon.
  \end{equation}
\end{proof}

We also recall the existence of complex geometrical optics solutions.
\begin{proposition} \label{prop:CGO}
  Let $n\in\{2,3\}$, $k>0$ and let $V$ be a qualitatively admissible
  potential.  Then there is $p = p(V,n) \geq 2$ and $c = c(V,R,k,n) <
  \infty$ with the following properties: if $\rho\in\C^n$ satisfies
  $\rho\cdot\rho+k^2=0$ and $\abs{\Im\rho} \geq c^{(n+1)/2}$ then
  there is $\psi\in L^p(\R^n)$ such that $u_0(x)=e^{\rho\cdot{x}}
  (1+\psi(x))$ solves $(\Delta+k^2(1+V))u_0=0$ in $\R^n$, and
  \[
  \norm{\psi}_{L^p(\R^n)} \leq c \abs{\Im\rho}^{-n/p-\beta}
  \]
  for some $\beta = \beta(V,n) > 0$. In addition there is the norm
  estimate $\norm{\psi}_{H^2(B_{2R})} \leq c\abs{\rho}^2$.
\end{proposition}

Proposition~\ref{prop:CGO} specializes Proposition 7.6 from \cite{BL2016}. Also,
mainly by Corollary 6.2 from that same paper, together with the use of
Taylor's theorem on the real-analytic incident wave $u^i$, we can show

\begin{proposition}
  Let $n\in\{2,3\}$ and let the background assumptions hold with $u^i$
  a normalized Herglotz wave. Let $V=\chi_\polygon \varphi$ be a
  qualitatively admissible potential. Choose coordinates such that the
  origin is a vertex of $\polygon$ where $\varphi\neq0$. Let $N\in\N$
  be such that $\partial^\gamma u^i(\bar0)=0$ for $\abs{\gamma}<N$ and
  set
  \[
  P_N(x) = \sum_{\abs{\gamma}=N} \frac{\partial^\gamma
    u^i(\bar0)}{\gamma!} x^\gamma.
  \]
  Let $\rho\in\C^n$ be such that it satisfies the assumptions of
  Proposition \ref{prop:CGO}, $\abs{\Re\rho} \geq \max(1,k)$ and
  $\Re\rho \cdot x \leq -\delta_0\abs{x}\abs{\Re\rho}$ for some
  $\delta_0>0$ and any $x\in\polygon$. Then
  \begin{equation} \label{upperBound}
  c \abs{ \int_\cone e^{\rho\cdot{x}} P_N(x) dx } \leq
  \abs{\Re\rho}^{-N-n-min(1,\alpha,\beta)} + \abs{\Re\rho}^3
  \sup_{\partial (\cone \cap B(\bar0,h))} \{\abs{u^s}, \abs{\nabla
    u^s}\}
  \end{equation}
  where $\cone$ is the open cone generated by $\polygon$ at the
  origin, $h=h(\polygon)$ is the minimal distance from any vertex of
  $\polygon$ to any of its non-adjacent edges, and the constant $c>0$
  depends on $V, N, \delta_0$ and $k$.
\end{proposition}

\bigskip
Next is the turn of a lower bound to the Laplace transform for
homogeneous harmonic polynomials of arbitrary degree. The proof is a
compactness argument with basis in the non-vanishingness proofs from
\cite{BPS} and \cite{PSV}. We recall that the norm for homogeneous
polynomials is
\[
\norm{P} = \int_{\mathbb{S}^{n-1}} \abs{P(\theta)} d\sigma(\theta).
\]

\begin{proposition}
  Let $n\in\{2,3\}$, $\cone \neq \emptyset$ be either an open orthant
  (3D) or an oblique open cone (2D). For $N\in\N$ set
  \[
  \mathcal P_N = \Big\{ P:\C^n\to\C \Big| \Delta P\equiv0, P(x) =
  \sum_{\abs{\gamma}=N} c_\gamma x^\gamma \Big\}.
  \]
  Let the angle of $\cone$ be at most $2\alpha_m < \pi$ and let
  $\alpha_m + \alpha_d < \pi/2$.  Then there is $\tau_0>0$ and $c>0$,
  both depending only on $\cone, N, n, \alpha_m+\alpha_d$ with the
  following properties: If $P\in\mathcal P_N$ then there is a curve
  $\tau\mapsto \rho(\tau)\in\C^n$ satisfying
  $\rho(\tau)\cdot\rho(\tau)+k^2=0$, $\tau=\abs{\Re\rho(\tau)}$,
  \[
  \Re\rho(\tau)\cdot{x} \leq -\cos(\alpha_m+\alpha_d)
  \abs{\Re\rho(\tau)} \abs{x}
  \]
  for all $x\in\cone$, and such that if $\tau\geq\tau_0$ then
  \begin{equation} \label{lowerBound}
  \abs{\int_\cone e^{\rho(\tau)\cdot{x}} P(x) dx} \geq
  \frac{c\norm{P}} {\abs{\Re\rho(\tau)}^{N+n}}.
  \end{equation}
\end{proposition}

\begin{proof}
  We identify $\mathcal{P}_N$ with a subset of $\C^m$, where $m =\# \{
  \gamma \in \N^n | \abs{\gamma} = N \} = (N + n - 1)! / (N!(n-1)!)$,
  by mapping $P\in\mathcal{P}_N$ to the point corresponding to its
  coefficients listed in some fixed order (e.g. by the lexical order
  of the multi-indices $\gamma$). This induces a topology on
  $\mathcal{P}_N$ which makes it a complete metric space. The space
  $\mathcal{P}_N \cap \{\norm{P}=1\}$ is compact.
  
  We will first consider the easier case of a complex vector
  satisfying $\zeta\cdot\zeta=0$ instead of
  $\rho\cdot\rho+k^2=0$. Write $\delta_0 = \cos (\alpha_m + \alpha_d)$
  and set
  \[
  R_{\cone, \delta_0} = \{ \zeta \in \C^n | \zeta \cdot \zeta = 0,
  \abs{\Re\zeta}=1, \Re\zeta\cdot{x} \leq - \delta_0 \abs{\Re\zeta}
  \abs{x} \forall x \in \cone \}.
  \]
  Also, write $\mathcal{L}P(\zeta) = \int_\cone \exp(\zeta\cdot{x})
  P(x) dx$ for $P \in \mathcal{P}_N$ and $\zeta \in R_{\cone,
  \delta_0}$. We claim first that
  \begin{equation}
    \inf_{P \in \mathcal{P}_N} \sup_{\zeta \in R_{\cone, \delta_0}}
    \abs{\mathcal{L}P(\zeta)} = c \norm{P} \label{infsup}
  \end{equation}
  for some constant $c = c (N, \cone, \delta_0) > 0$. By dividing $P$
  with $\norm{P}$ and the linearity of $\mathcal{L}$ we may assume
  that $\norm{P} = 1$. If \eqref{infsup} did not hold then for any $j
  \in \N$ there is $P_j \in \mathcal{P}_N$, $\norm{P_j}=1$ such that
  $\abs{\mathcal{L}P_j(\zeta)}<j^{- 1}$ for any $\zeta \in R_{\cone,
    \delta_0}$. Since $\mathcal{P}_N \cap \{\norm{P}=1\}$ is compact
  there is $P_{\infty} \in \mathcal{P}_N$, $\norm{P_{\infty}} = 1$ and
  a subsequence $P_{j_{\ell}} \rightarrow P_{\infty}$. Let $\zeta \in
  R_{\cone, \delta_0}$.  It is easily seen that
  $\abs{\mathcal{L}(P_{j_{\ell}} - P_{\infty})(\zeta)} \leq
  (N+n-1)!\delta_0^{1-N-n} \norm{P_{j_{\ell}}-P_{\infty}} \rightarrow
  0$ as $\ell \rightarrow \infty$. Hence
  $\abs{\mathcal{L}P_{\infty}(\zeta)}=0$ for any complex vector $\zeta
  \in R_{\cone, \delta_0}$, but this contradicts the Laplace transform
  lower bounds from \cite{BPS} and \cite{PSV}. Thus the lower bound
  (\ref{infsup}) holds, but for vectors satisfying
  $\zeta\cdot\zeta=0$.
  
  Let us build $\rho(\tau)$ by using a $\zeta$ from the previous
  paragraph. Let $P \in \mathcal{P}_N$ be arbitrary and take $\zeta
  \in R_{\cone,\delta_0}$ such that $\abs{\mathcal{L}P(\zeta)} \geq
  c\norm{P}/2$. For $\tau > 0$ set
  \[
  \rho(\tau) = \tau \Re\zeta + i \sqrt{\tau^2 + k^2} \Im\zeta.
  \]
  Then $\rho(\tau)/\tau \rightarrow \zeta$ as $\tau \rightarrow
  \infty$ and moreover $\rho(\tau)\cdot\rho(\tau)+k^2 = 0$, and
  $\Re\rho(\tau)\cdot{x}\leq -\delta_0\abs{\Re\rho(\tau)}\abs{x}$ for
  $x\in\cone$.  When $\tau$ is large enough we will have
  $\abs{\mathcal{L}P(\rho(\tau)/\tau)} \geq c\norm{P}/4$. The proof is
  as follows: set
  \[
  f(r) = \exp((\Re\zeta + ir\Im\zeta)\cdot{x}).
  \]
  Then $f(1) = \exp(\zeta\cdot{x})$ and
  $f\left(\sqrt{1+k^2/\tau^2}\right) = \exp(\rho(\tau)\cdot{x}/\tau)$.
  By the mean value theorem
  \[
  \abs{f(1) - f\left(\sqrt{1+k^2/\tau^2}\right)} \leq
  \sup_{1<r<\sqrt{1+k^2/\tau^2}} \abs{f'(r)}
  \abs{\sqrt{1+k^2/\tau^2}-1}.
  \]
  But note that $\sqrt{1+k^2/\tau^2}-1 = \tau^{-1}k^2/\left( \tau +
  \sqrt{\tau^2+k^2} \right) \leq k/\tau$. Also $f'(r) =
  i\Im\zeta\cdot{x} f(r)$ and since $\abs{\Re\zeta} = \abs{\Im\zeta} =
  1$ we get $\abs{f'(r)} \leq \abs{x} \exp(-\delta_0 \abs{x})$. In
  other words
  \[
  \abs{f(1)-f\left(\sqrt{1+k^2/\tau^2}\right)} \leq \frac{k}{\tau}
  \abs{x} e^{-\delta_0\abs{x}}.
  \]
  Finally we see the claim:
  \begin{align*}
    \Big\lvert \mathcal{L}P(\zeta) &- \mathcal{L}P
      \left(\frac{\rho(\tau)}{\tau}\right) \Big\rvert = \abs{\int_{\cone}
      \left(f(1)-f\left(\sqrt{1+k^2/\tau^2}\right)\right) P(x) dx}\\ &
    \leq \frac{k}{\tau} \int_\cone e^{-\delta_0\abs{x}} \abs{x}
    \abs{P(x)} dx = \norm{P} \frac{k}{\tau} \int_0^{\infty}
    e^{-\delta_0 r} r^{1+N+n-1} dr\\ & = (N+n)!  \delta_0^{-N-n} k
    \tau^{-1} \norm{P},
  \end{align*}
  and so $\abs{\mathcal{L}P(\rho(\tau)/\tau)} > c\norm{P}/4$ if
  $\tau>4(N+n)! \delta_0^{-N-n} k/c$. A change of variables gives then
  $\mathcal{L}P(\rho(\tau)/\tau) = \tau^{N+n} \mathcal{L}P(\rho
  (\tau))$ and so the proposition is proven.
\end{proof}

\section{Bound for far-field pattern with incident Herglotz wave} \label{sect:ffbound}

\begin{proof}[Proof of Theorem \ref{lowerBoundThm}]
  Let $\mathcal{S}=\mathcal{S}(V,k)$ be such that
  $\norm{u^s}_{H^2(B_{2R})} \leq \mathcal{S}$ whenever the incident
  wave is a normalized Herglotz wave. Let $u^i$ be a normalized
  incident wave and $u^s$ the corresponding scattered wave. Let $u^i$
  be of order $N\in\N$ at the vertex $x_c$, which we may take as being
  the origin, and on which $\varphi\neq0$. Moreover let $P_N$ be its
  $N$-th degree homogeneous Taylor polynomial at $\bar0$. Note that
  this polynomial is harmonic because $(\Delta+k^2)u^i=0$. Firstly
  combine \eqref{lowerBound}, \eqref{upperBound} and
  \eqref{boundaryBound} to get
  \[
  c\norm{P_N} \leq \abs{\Re\rho(\tau)}^{-\min(1,\alpha,\beta)} +
  \frac{\abs{\Re\rho(\tau)}^{N+n+3}}{\sqrt{\ln \ln
      \frac{\mathcal{S}}{\norm{u^s_{\infty}}_{L^2(\mathbb{S}^{n-1})}}}}
  \]
  when $\norm{u^s_{\infty}} \leq \varepsilon_m$ and $\tau \geq \tau_0$,
  with constants depending on $V, N, n, k, \alpha_m + \alpha_d,
  \mathcal{S}$.

  The estimate above depends monotonically on each individual
  constant. Fix $\mathcal{N}\in\N$ and set
  \[
  \varepsilon_{m, \mathcal{N}} = \min_{N \leq \mathcal{N}}
  \varepsilon_m, \quad \tau_{0, \mathcal{N}} = \max_{N \leq
    \mathcal{N}} \tau_0, \quad c_{\mathcal{N}} = \min_{N \leq
    \mathcal{N}} c.
  \]
  Then if $\mathcal{N}\geq N$ the estimate holds with these new
  constants and $\mathcal{N}$ in the exponent instead of $N$ (since
  $\abs{\Re\rho(\tau)} = \tau \geq 1$). In other words
  \begin{equation} \label{toOptimize}
    c_{\mathcal{N}}\norm{P_N} \leq \abs{\Re\rho(\tau)}^{-\min(1,
      \alpha, \beta)} + \frac{\abs{\Re\rho(\tau)}^{\mathcal{N}+ n +
        3}}{\sqrt{\ln \ln \frac{\mathcal{S}}{\norm{u^s_{\infty}}_{L^2
            (\mathbb{S}^{n-1})}}}}
  \end{equation}
  when $\norm{u^s_{\infty}} \leq \varepsilon_{m, \mathcal{N}}$ and
  $\tau \geq \tau_{0,\mathcal{N}}$ and $u^i$ is of order $N \leq
  \mathcal{N}$ at $\bar{0}$.
  
  Write $\gamma = \min (1, \alpha, \beta)$ and $R = \sqrt{
    \ln\ln(\mathcal{S}/\norm{u^s_{\infty}
    }_{L^2(\mathbb{S}^{n-1})})}$. The right-hand side of
  \eqref{toOptimize} has a global minimum at the point
  \[
  \tau_m = (\gamma R/(\mathcal{N}+n+3))^{1/(\mathcal{N}+n+3+\gamma)},
  \]
  and the minimal value there is given by $c(\mathcal{N},n,\gamma)
  R^{-\gamma/(\mathcal{N}+n+3+\gamma)}$. Hence if $\tau_m \geq
  \tau_{0, \mathcal{N}}$, we may set $\tau=\tau_m$ in
  \eqref{toOptimize} and solve for the norm of the far-field
  pattern. We then have
  \begin{equation} \label{farfieldBound1}
    \norm{u^s_{\infty}}_{L^2(\mathbb S^{n-1})} \geq \frac{\mathcal{S}}{\exp \exp \big(c
      \norm{P_N}^{-\ell}\big)}
  \end{equation}
  where the exponent $\ell \geq 2 (\mathcal{N}+n+4)$ and $c < \infty$
  may be chosen to depend only on $V, n, k, \mathcal{N}$. The other
  case, namely $\tau_m < \tau_{0,\mathcal{N}}$ reduces to
  $\norm{u^s_{\infty}}_{L^2(\mathbb S^{n-1})} > \mathcal{S} / (\exp \exp c)$ for some $c = c
  (V, n, k, \mathcal{N})$.
\end{proof}

\section{Vanishing of the interior transmission eigenfunction at corners} \label{sect:vanishing}

\begin{proof}[Proof of Theorem \ref{vanishingThm}]
  Let us start by taking a sequence of incident Herglotz waves
  \[
  v_j(x) = \int_{\mathbb S^{n-1}} \exp(i k \theta\cdot{x}) g_j(\theta)
  d\sigma(\theta)
  \]
  approximating the interior transmission eigenfunction $v$ in the
  $L^2(\polygon)$-norm; see Theorem~\ref{thm:herg1}. We may assume for example that
  $\norm{v-v_j}_{L^2(\polygon)} < 2^{-j}$. By Proposition
  \ref{prop:smallFF} we have the estimate
  \begin{equation} \label{FF20}
    \norm{ v^s_{j\infty} }_{L^2(\mathbb S^{n-1})} < C_{V,k} 2^{-j}
  \end{equation}
  for the corresponding far-field pattern.  The assumption on $v$
  allows us to have $\norm{g_j}_{L^2(\mathbb S^{n-1})} \leq G <
  \infty$ for all $j$.

  Let $x_c\in\partial\polygon$ be a vertex such that
  $\varphi(x_c)\neq0$. Our goal is to estimate the integral of
  $\abs{v}$ in $B(x_c,r) \cap \polygon$. We will achieve that by
  estimating the corresponding integrals of $v_j$. Let us denote $B =
  B(x_c,r)$ for convenience. Let $N_j$ be the order of $v_j$ at $x_c$,
  so $\partial^\alpha v_j (x_c) = 0$ for $\abs{\alpha} < N_j$. Then by
  the smoothness of $v_j$ we have $N_j \in \N \cup \{\infty\}$. By its
  real-analyticity we have $N_j<\infty$. Fix $N \in \N$. If $N_j\geq
  N$, then
  \[
  \norm{v}_{L^1 (B\cap\polygon)} \leq
  \norm{v-v_j}_{L^1(B\cap\polygon)} + \norm{v_j}_{L^1(B)} \leq
  C_{\polygon} 2^{-j} + C_{N,v_j} r^{N+n}.
  \]
  The theorem would follow if $N_j \geq 1$ for an inifinite sequence
  of $j$'s and $\sup_j C_{N,v_j} < \infty$ for these.

  Let us study $\norm{v_j}_{L^1}$ in more detail. Again, assuming
  $N_j\geq N$, by Taylor's theorem
  \[
  v_j(x) = \sum_{\abs{\alpha}=N} \frac{\partial^{\alpha}
    v_j(x_c)}{\alpha!} (x-x_c)^{\alpha} + R_{v_j,N,x_c}(x).
  \]
  Set $P_{j,N}(x) = \sum_{\abs{\alpha}=N} \partial^{\alpha} v_j(x_c)
  x^{\alpha} / \alpha!$, and so $v_j(x) = P_{j,N}(x-x_c) +
  R_{v_j,N,x_c}(x)$. Define $\norm{P_{j,N}} = \int_{\mathbb S^{n-1}}
  \abs{P_{j,N}(\theta)} d\sigma(\theta)$. Then
  \[
  \norm{P_{j,N}(\cdot-x_c)}_{L^1(B)} = \frac{\norm{P_{j,N}}}{N+n}
  r^{N+n}
  \]
  and
  \begin{align*}
    \lvert R_{v_j,N,x_c}&(x)\rvert \leq \sum_{\abs{\beta}=N+1}
    \frac{\abs{x-x_c}^{N+1}}{\beta!} \max_{\abs{\gamma}=N+1}
    \max_{\abs{y-x_c}\leq1} \abs{\partial^{\gamma} v_j(y)}\\ & \leq
    C_{N,n} \abs{x-x_c}^{N+1} \max_{\abs{\gamma}=N+1}
    \max_{\abs{y-x_c}\leq1} \int_{\mathbb S^{n-1}} k^{N+1}
    \abs{\theta^{\gamma}} \abs{g_j(\theta)} d\sigma(\theta)\\ & \leq
    C_{N,k,n} \abs{x-x_c}^{N+1} \norm{g_j}_{L^2(\mathbb S^{n-1})}.
  \end{align*}
  In other words $\norm{v_j}_{L^1(B)} \leq C_{N,k,n,G} (\norm{P_{j,N}}
  + r) r^{N+n}$ if $v_j$ has order $N_j \geq N$ at $x_c$ since we had
  assumed the uniform bound $\norm{g_j}_{L^2(\mathbb S^{n-1})} \leq
  G$. Thus
  \begin{equation} \label{vOrderBound}
    \norm{v}_{L^1(B\cap\polygon)} \leq C_{\polygon} 2^{-j} +
    C_{N,k,n,G} (\norm{P_{j,N}}+r) r^{N+n}
  \end{equation}
  whenever $N_j \geq N$.
  
  Fix $N=1$ now. At least one of the following is true: 1) there is a
  subsequence of $v_j$ for which $N_j \geq 1$, or 2) there is a
  subsequence for which $N_j = 0$. In the former case we note that
  $\norm{P_{j,1}} \leq C_{n,k,G} < \infty$ by the Herglotz wave
  formula for $v_j$, and thus \eqref{vOrderBound} implies that $v$ has
  order $1$ at $x_c$; a stronger result than in the theorem. So
  consider case 2) from now on.

  We may assume that $N_j = 0$ for all $j$ since we are in case 2). We
  will use Theorem \ref{lowerBoundThm}. To use \eqref{FFlowerBound} we
  need to have normalized incident Herglotz waves, a property which is
  not necessarily true for $v_j$. However note that $v_j /
  \norm{g_j}_{L^2(\mathbb S^{n-1})}$ is normalized. We have
  \[
  \norm{v_j}_{L^2(\polygon)} \geq \norm{v}_{L^2(\polygon)} -
  \norm{v-v_j}_{L^2(\polygon)} > 1 - 2^{-j}
  \]
  and
  \begin{align*}
    \norm{v_j}_{L^2(\polygon)} &\leq \int_{\mathbb S^{n-1}}
    \norm{e^{ik\theta\cdot x}}_{L^2(\polygon,x)} \abs{g_j(\theta)}
    d\sigma(\theta) \\ &\leq \sqrt{m(\polygon) \sigma(\mathbb
      S^{n-1})} \norm{g_j}_{L^2(\mathbb S^{n-1})}.
  \end{align*}
  In other words $\norm{g_j}_{L^2(\mathbb S^{n-1})} \geq 1 /
  \left(2\sqrt{m(\polygon) \sigma(\mathbb S^{n-1})} \right) > 0$ when
  $j\geq1$. We also know that $v_j$ has order $0$ at $x_c$. Hence by
  Theorem \ref{lowerBoundThm}
  \[
  \norm{v^s_{j \infty}} \geq \frac{\mathcal{S} \norm{g_j}_{L^2(\mathbb
      S^{n-1})}}{\exp \exp c \min (1,
    \frac{\norm{P_{j,0}}}{\norm{g_j}_{L^2(\mathbb S^{n-1})}})^{-\ell}}
  \geq \frac{\mathcal{S}/ \left(2 \sqrt{m(\polygon) \sigma(\mathbb
      S^{n-1})} \right)}{\exp \exp c \min(1,
    \frac{\norm{P_{j,0}}}{G})^{-\ell}}
  \]
  for all $j$. By (\ref{FF20}) and the above we see that
  $\norm{P_{j,0}} \rightarrow 0$ as $j \rightarrow \infty$.

  By having $N=0$ in (\ref{vOrderBound}) and taking the limit
  $j\rightarrow\infty$ we see that $\norm{v}_{L^1(B)} \leq C_{k,n,G}
  r^{n+1}$. Hence
  \[
  \lim_{r\rightarrow0} \frac{1}{m(B)} \int_{B} \abs{v(x)} dx = 0.
  \]
\end{proof}

\section{Discussion}\label{sect:discussion}

In this paper, we are concerned with the transmission eigenvalue
problem, a type of non elliptic and non self-adjoint eigenvalue
problem. We derive intrinsic properties of transmission eigenfunctions
by showing that they vanish near corners at the support of the
potential function involved. This is proved by an indirect approach,
connecting to the wave scattering theory. Indeed, we first show that
by using the Herglotz-approximation of a transmission eigenfunction as
an incident wave field, the generated scattered wave can have an
arbitrarily small energy in its far-field pattern. On the other hand,
we establish that with an incident Herglotz wave the scattered
far-field pattern has a positive lower bound depending on the Herglotz
wave's order of vanishing at a corner. This hints that the
transmission eigenfunction should vanish near the corner
point. Nevertheless, the rigorous justification of the vanishing
property is a highly nontrivial procedure.

To our best knowledge, Theorem~\ref{vanishingThm} is the first result
in the literature on the intrinsic properties of transmission
eigenfunctions. The vanishing behaviour obviously carries geometric
information of the support of the involved potential function
$V$. Indeed, in inverse scattering theory, an important problem
arising in practical application is to infer knowledge of $V$ by
measurements of the far-field pattern
$u^s_\infty\left(\frac{x}{\abs{x}}; u^i\right)$
(cf. \cite{CK,Isa,Nachman88,Sylvester--Uhlmann,Uhl2,Uhl}). There is
relevant study on determining the transmission eigenvalues using
knowledge of $u^s_\infty\left(\frac{x}{\abs{x}}; u^i\right)$
(cf. \cite{CH2013inBook}). Clearly, it would be interesting and useful
as well to determine the corresponding eigenfunctions from the inverse
scattering point of view. Indeed, as suggested by
Theorem~\ref{vanishingThm}, if the unknown function $V$ is supported
in a convex polyhedral domain, then one might use the vanishing
property of the corresponding transmission eigenfunction to determine
the vertices of the polyhedral support of $V$. As mentioned earlier,
in the upcoming numerical paper \cite{BLLW}, we shall show that the
vanishing order is related to the angle of the corner and the
vanishing behaviour also occurs at the edge singularities of
$\mathrm{supp}(V)$. Hence, one can use these intrinsic properties of
transmission eigenfunctions to determine the polyhedral support of an
unknown function $V$. This is beyond the aim and scope of the present
article and we shall investigate this interesting issue in our
upcoming papers.

We will comment on the requirement of uniformly bounded Herglotz
kernels of Theorem~\ref{vanishingThm}. It is a technical condition and
very difficult to relate directly to Theorem~\ref{thm:herg1}. This
study is a first step in the research of intrinsic properties of
transmission eigenfunctions and we have brought a new phenomenon into
attention. This observation was derived from the apparent
contradiction of the well-known Theorem~\ref{thm:herg1} and our new
Theorem~\ref{lowerBoundThm}. In addition, the upcoming numerical study
\cite{BLLW} gives evidence that this vanishing phenomenon is true more
generally. Also in another upcoming paper (Proposition 3.5 in
\cite{BL2017}) we study corner scattering with more general incident
waves, namely waves in $H^2$ that do not need to be defined outside a
small interior neighbourhood of a corner of $\polygon$. That result
suggests that the condition of approximation by uniformly bounded
kernels can be swapped out for the condition that $v$ restricted to
$\polygon \cap B(x_c,\varepsilon)$ is in $H^2$. In other words, if a
transmission eigenfunction is smooth enough near a corner, then it
must vanish at that corner. We shall further explore this interesting
issue in forthcoming papers.

Finally, we would like to mention that Theorem~\ref{lowerBoundThm} is
of significant interest for its own sake, particularly for
invisibility cloaking (cf. \cite{GKLU4,GKLU5}). Indeed, it generalises
our earlier corner scattering result in \cite{BL2016} where the
incident wave fields are confined to be plane waves. It suggests that
if the support of the underlying scatterer possesses corner
singularities, then in principle for any incident fields, invisibility
cannot be achieved. On the other hand, it also suggests that if one
intends to diminish the scattering effect, then the incident wave
field should be such chosen that it vanishes to a high order at the
corner point. This is another interesting topic worth of further
investigation, especially the corresponding extension to anisotropic
scatterers.

\section*{Acknowledgement}

We are grateful to Professor Fioralba Cakoni for helpful discussion on Proposition~\ref{prop:smallFF}
which inspires this article. The work of H Liu was supported by the
FRG fund from Hong Kong Baptist University, the Hong Kong RGC grant
(No.\, 12302415) and NSF of China (No.\,11371115).

\end{document}